\documentclass[12pt]{article}
\usepackage[centertags]{amsmath}     
\usepackage{amssymb}
\usepackage{amsthm}
\usepackage[alphabetic, initials]{amsrefs}
\usepackage{verbatim}                
\usepackage[dvips]{graphics}
\usepackage{enumerate}

\numberwithin{equation}{section}

\newtheorem{lemma}[equation]{Lemma}

\newtheorem{conj}[equation]{Conjecture}

\newtheorem{corollary}[equation]{Corollary}

\newtheorem{theorem}[equation]{Theorem}

\newtheorem*{thm1.3p}{Theorem 1.3'}
\newtheorem*{thm1.3pp}{Theorem 1.3''}

\newcommand{\remark}{\noindent\textbf{Remark.}}

\newcommand{\me}{\mathrm{e}}

\newcommand{\eps}{\varepsilon}

\newcommand{\vphi}{\varphi}

\newcommand{\Cb}{\mathbb{C}}

\newcommand{\Rb}{\mathbb{R}}

\newcommand{\Bl}{\mathcal{B}}

\newcommand{\dist}{\mathrm{dist}}

\newcommand{\Hess}{\mathrm{Hess}}
\newcommand{\Ric}{\mathrm{Ric}}
\newcommand{\Sec}{\mathrm{Sec}}
\newcommand{\inj}{\mathrm{inj}}
\newcommand{\Vol}{\mathrm{Vol}}

\newcommand{\Det}{\mathrm{det}}

\newcommand{\grad}{\mathrm{grad}}

\newcommand{\dif}[1]{{\;d #1}}

\makeatletter
\def\blfootnote{\xdef\@thefnmark{}\@footnotetext}
\makeatother

\begin{document}
\title{The effect of curvature on convexity properties
of harmonic functions and eigenfunctions}
\author{Dan Mangoubi}
\date{\textit{Dedicated to Shmuel Agmon with admiration and gratitude \\
on the occasion of his 90th birthday}}
\maketitle
\begin{abstract}
We give a proof of the
growth bound of Laplace-Beltrami eigenfunctions due to Donnelly
and Fefferman which is 
probably the easiest and the most elementary one.
Our proof also gives new quantitative geometric estimates
in terms of curvature bounds
which improve and simplify 
previous work by Garofalo and Lin.
The proof is based on a generalization of a convexity property of harmonic functions in $\Rb^n$ to harmonic functions
on Riemannian manifolds following Agmon's ideas.
\end{abstract}

\section{Introduction}
In their seminal paper~\cite{don-fef88} Donnelly and Fefferman 
found growth bounds (DF-growth bound) for 
eigenfunctions on compact Riemannian manifolds.
Roughly, they showed that a 
$\lambda$-eigenfunction grows like a polynomial
of order $\sqrt{\lambda}$ at most.
This result is central in the study of eigenfunctions.
In~\cite{don-fef88} it was applied to prove Yau's conjecture
on real analytic manifolds. Namely, sharp upper
and lower bounds on the size of the nodal set on real analytic 
manifolds were found.
The proof of the growth bound in~\cite{don-fef88} 
went through a fine version of a Carleman type inequality
for the operator $\Delta+\lambda$, with a
careful geometric choice of the weight function.

Recently after, Lin (\cite{lin91}), based on an 
earlier work with Garofalo (\cite{gar-lin86}),
 gave a simpler proof of the growth bound. 
This proof is based on properties of
the spherical $L^2$-norm, $q(r)$ (defined in~(\ref{def:q})),
for harmonic functions. It had been known (\cites{agmon66, almgren79})
that in $\Rb^n$, $\log q$ is monotonically increasing and convex
as a function of $\log r$.
Equivalently, 
$rq'(r)/q(r)$ is monotonically increasing.
Garofalo-Lin showed that for a harmonic function 
defined on a general Riemannian manifold
$\me^{\Lambda r}rq'/q$ is monotonically increasing in (0, R),
where $\Lambda$ and $R$ are some positive constants
depending on bounds on the Riemannian metric, on its first derivatives
and  on the ellipticity constant of the Riemannian metric.
This result can be viewed as an approximated convexity result.
The proof of this result was based on a non-trivial geometric
variational argument which was first used by Almgren~\cite{almgren79}.

The first aim of this paper 
is to give new geometric estimates on 
$\Lambda$ and $R$ in terms of the curvature of the manifold.
Namely, we find that all one needs is a lower and an upper bound on the
sectional curvature in order to guarantee the existence of 
$\Lambda$ and $R$. Moreover, we show that in fact
$\me^{C_1r^2K}rq'(r)/q(r)$ is monotonic in $(0, R)$, where $K$ is an upper bound on the
curvature, $R$ is the minimum of $C_2/\sqrt{K^+}$ and the injectivity radius, 
and $C_1, C_2$ depend only on the dimension of the manifold.
We emphasize that our result distinguishes between negative and positive curvatures. 
This is the content of the main 
Theorem~\ref{thm:perturbed-convexity}.

The second aim of this paper is to have a simple proof of the DF-growth 
bound for  eigenfunctions. Due to the importance of this result
three simplifications to its proof had been previously given 
by different authors in the course of years, which we briefly survey:

The idea of Lin in~\cite{lin91}
 was to consider
  a conic manifold, $N$, over $M$ 
and to extend the eigenfunction $u_\lambda$
  to a harmonic function on $N$. Then, Lin applied the monotonicity
property of $\me^{\Lambda r}rq'/q$ from~\cite{gar-lin86} for 
the harmonic function obtained, and went back to the eigenfunction.

 Jerison and Lebeau
applied in~\cite{jer-leb96} a similar extension of eigenfunctions.
Then, they could use standard Carleman type inequalities for harmonic functions,
instead of the original approach taken by Donnelly and Fefferman
in which a special and delicate Carleman type inequality for
eigenfunctions was used.

In dimension two Nazarov-Polterovich-Sodin~\cite{naz-pol-sod05} 
took advantage of the conformal coordinates, thus letting them to simplify
the problem by considering only the standard Laplace operator in $\Rb^2$. 
Then, they extend the eigenfunction to a harmonic function on 
$N=M\times\Rb$,
and apply convexity argument on the harmonic function (in $\Rb^3$) obtained.
Their proof of convexity of $\log q$ is considerably simpler
than the variational approach taken in~\cite{gar-lin86}.
It is close in spirit to Agmon's approach.
This gives the easiest proof of the DF-growth bound in dimension two,
since no need for variational arguments or Carleman type inequalities
at all is required.

This paper extends the work started in~\cite{naz-pol-sod05}, 
to dimensions
 $\geq 3$, where no conformal coordinates exist.
We follow and generalize Agmon's ideas in~\cite{agmon66},
where a general approximated convexity theorem for
second order elliptic equations is proved
by considering them as an abstract second order ODE.
Our contribution here comes in adding the geometric point of view,
clarifying the way curvature affects 
the Euclidean result. Our proof also simplifies 
and improves Agmon's results in~\cite{agmon66}.
In this way we are able circumvent the need to  use 
the non-trivial variational argument in~\cite{gar-lin86}
or any Carleman type inequality.
\vspace{1Ex}

\noindent\textbf{Organization of the paper.}
The main result is presented in section~\ref{sec:main-thm}.
In section~\ref{sec:extension} we recall a way eigenfunctions
can be extended to harmonic functions and the translation of the convexity
property of harmonic functions to a local growth bound on eigenfunctions.
In section~\ref{sec:global-growth} we conclude the proof of
the DF-growth bound on compact manifolds.
and we outline the proof of Yau's conjecture in~\cite{don-fef88}. Sections~\ref{sec:extension} and~\ref{sec:global-growth} are strongly based
on~\cite{naz-pol-sod05}.
In section~\ref{sec:proof-thm} we give the proof of the main theorem.
In section~\ref{sec:examples} we consider constant curvature manifolds
as examples to the main theorem and find a second proof in some of these cases.
In section~\ref{sec:discussion} we discuss several open questions.
\vspace{1Ex}

\noindent\textbf{Notation.}
Throughout this paper $C_i, C_i(n)$ denote positive constants which depend
only on dimension. The positive constants $C_g(\ldots)$ depend on bounds
on the metric $g$, its first derivatives, its ellipticity constant
and additional parameters appearing in parentheses.
\vspace{1Ex}

\noindent\textbf{Acknowledgements.}
I am grateful to Leonid Polterovich and Misha Sodin
for encouraging me to write this paper and for several discussions
concerning it. 
I thank J\'ozef Dodziuk for pleasurable relevant discussions.
I owe my gratitude to Shing-Tung Yau for 
his support and for stimulating questions which left their imprint
on this paper. 
Finally, I would like to thank the anonymous referee for his valuable
comments. This research was partially supported by ISF 
grant no.\ 225/10 and by BSF grant no.\ 2010214.

\section{Main Theorem: A perturbed $\log$-convexity 
property of harmonic functions}
\label{sec:main-thm}
Let $u$ be a harmonic function in $\Rb^n$.
Let $q(r)$ denote the square of the spherical $L^2$-norm: 
$$q(r):=\int_{S_r} u^2 \dif \sigma_r\ ,$$
where $S_r$ denotes the sphere of radius $r$ centered at $0$, 
and $d\sigma_r$ is the standard area measure on $S_r$. 
It's easy to check that $q$ is a convex function of $\log r$.
It turns out that even $\log q$ is  a convex function of $\log r$:
\begin{theorem}[\cite{agmon66}]
\label{thm:log-convexity-rn}
$q$ has the following two properties:
\begin{enumerate}
\item[(i)]
$\displaystyle{q'(r)\geq \frac{n-1}{r}q(r)}$,
\item[(ii)]
$\displaystyle{q''(r)+\frac{1}{r}q'(r)-\frac{q'(r)^2}{q(r)}\geq 0}$.
\end{enumerate}
\end{theorem}
In dimension $2$ this can be seen by a complex analysis argument.
In higher dimensions this fact goes back at least to Agmon~(\cite{agmon66}),
and it was rediscovered by Almgren~\cite{almgren79}. 
Landis (\cite{landis63}*{Ch. II.2}) 
found also several results close in spirit to that one.
All these kinds of results were 
inspired by Hadamard's Three Circles Theorem (See~\cite{ahlfors-ca}*{Ch. 6.2}),
which shows log-convexity of the spherical $L^\infty$-norm for
a holomorphic function.
\vspace{1ex}

\noindent\textbf{Remark.} It is somewhat surprising that the 
fundamental solution does not play a role here:
$\log q$ is a convex function of $\log r$ in \emph{all}
dimensions. The weaker statement
is that $\log q$ is a convex function of $G(r)=-1/r^{n-2}$,
which is equivalent to
$q\Delta \log q= q''(r)+\frac{n-1}{r}q'(r)-\frac{q'(r)^2}{q(r)}\geq 0$. 
\vspace{1ex}

When considering harmonic functions on manifolds, 
one expects a perturbed version of Theorem~\ref{thm:log-convexity-rn}
in small geodesic balls. However, it is not clear a priori 
how far from the center this perturbation goes and how curvature controls it. Theorem~\ref{thm:perturbed-convexity} below will give an answer to these questions.
Let $u$ be a harmonic function defined in a small geodesic ball
of a Riemannian manifold $N$. Let 
\begin{equation}
\label{def:q}
q(r):= \int_{S(r)} u^2\, dA_r\ ,
\end{equation}
where $S(r)$ is a geodesic sphere centred at $p\in N$, 
and $dA_r$ is the area form on $S(r)$.
$\sqrt{q}$ is the spherical 
$L^2$-norm on a geodesic sphere of radius $r$.
We let $\Sec_N$ denote the sectional curvature of $N$, 
$K^+=\max\{K, 0\}$, 
$$
\sin_K r = \left\{\begin{array}{lcl}
\frac{\sin(r\sqrt{K})}{\sqrt{K}} &,& K>0,\\
        r&,& K=0,\\
        \frac{\sinh(r\sqrt{-K})}{\sqrt{-K}} &,& K<0\ .
        \end{array}\right.
$$
and $\cot_K r=(\sin_K r)'/(\sin_K r)$.
We can now state our main result:
\begin{theorem}
\label{thm:perturbed-convexity}
Let $N$ be a Riemannian manifold. Let $u$ be a harmonic function
on a geodesic ball in $N$, and $q$ defined as in~(\ref{def:q}).
Let $\kappa, K\in\Rb$, $\kappa\leq K$. Let $R=\min\left(\inj(M), \pi/(2\sqrt{K^+})\right).$
We have
\begin{enumerate}
\item[(i)] If $\Sec_N\leq  K$ 
then $q(r)/(\sin_K r)^{n-1}$ is monotonically increasing for \mbox{$r<R$}. Equivalently, 
 $$(\log q)'(r) \geq (n-1)(\cot_K r)\ .$$
\item[(ii)] 
If $\kappa\leq \Sec_N \leq K$
then for $r<R$
\begin{multline*}
(\log q)''(r)+(\cot_K r)(\log q)'(r)
+(n+1)(\cot_\kappa r-\cot_K r)(\log q)'(r) \\ 
\geq -K-(n-2)K^{+}-(2n-3)(K-\kappa)\ .
\end{multline*}
\end{enumerate}
\end{theorem}
The proof of the theorem is given in Section~\ref{sec:proof-thm}.
\vspace{1ex}

\noindent\textbf{Remarks:}
\begin{itemize} 
\item It looks like in dimensions $n\geq 3$ the result
for negative curvature is better. However, this seems to be an
 artificial phenomenon since one could state part~(ii) of the theorem
with the function $\tilde{q}=q/(\sin_K r)^{n-1}$ replacing $q$:
Then, the RHS becomes
$-(n-2)K^{-}-(2n-3)(K-\kappa)$ which gives ``advantage'' to positive
curvature in dimensions $n\geq 3$ (see also the discussion in 
Section~\ref{subsec:orthogonal}).
  
\item For the constant non-positive (non-negative) curvature case 
 we get an exact convexity statement for $\log q$ (for $\log\tilde{q}$). 
\item Comparing to the result of Garofalo and Lin in~\cite{gar-lin86}, 
from part~(ii) one deduces that 
$\me^{6nr^2K}rq'(r)/q(r)$ is monotonically increasing 
for  $r<R$. Observe that
besides the explicit estimates of 
$\Lambda$ and $R$ mentioned in the introduction
this gives also a correction of the result in~\cite{gar-lin86}
in the power of $r$ in the exponential
term. Moreover, the statement here is more geometric in nature.
\end{itemize}

We now would like to have an integrated version of 
Theorem~\ref{thm:perturbed-convexity}. We restrict our attention
only to the case $\kappa=-K$, $K>0$.
We obtain a local doubling estimate for harmonic functions (see proof in Section~\ref{subsec:cor}).
\begin{corollary}
\label{cor:integrated-convexity}
Let $N$ be a complete Riemannian manifold of dimension $n$ with 
$|\Sec_N|\leq K$. Then
$$   \frac{q(2r)}{q(r)}\leq 
   \left(\frac{q(2s)}{q(s)}\right)^{1+32n r^2K}$$
for all $r<s<1/(4\sqrt {nK})$.
\end{corollary}
\section{Harmonic extension of Eigenfunctions}
\label{sec:extension}
In this section we recall a connection between harmonic functions and eigenfunctions found in~\cites{lin91, jer-leb96, naz-pol-sod05}.
Let $M$ be a Riemannian manifold of dimension $m$.
Let $u_\lambda$ be a $\lambda$-eigenfunction on $M$.
Consider the direct product Riemannian manifold $N=M\times\Rb$
of dimension $n=m+1$,
where the metric on $\Rb$ is the standard one.
Let $H$ be the following function on $N$:
$$\forall x\in M, t\in\Rb\ \ H(x, t) := u_\lambda(x)\cosh(\sqrt{\lambda}t)\ .$$
$H$ extends $u_\lambda$ to $N$ and is harmonic on $N$, since
the Laplacian on $N$ can be written as
$$\Delta_N u =\Delta_M u + \frac{\partial^2u}{\partial t^2}\ .$$

On $N$ we take geodesic coordinates $(r, \theta_1,\ldots, \theta_{n-1})$
in a neighborhood of the point $(p, 0)\in N$.
In these coordinates the metric $g_N$ takes  
the following form 
$$g_N=\dif r^2+r^2 a_{ij}\dif\theta^i\dif\theta^j\quad 1\leq i, j\leq n-1\, .$$
We let $\hat{\theta} = (\theta_1, \ldots, \theta_{m-1})$,
and $b_{ij}(r, \hat{\theta}):=a_{ij}(r,\hat{\theta}, 0)$.
$$g_M= dr^2 +r^2b_{ij}\dif\theta^i \dif\theta^j,\quad 1\leq i,j\leq m-1\ .$$
Accordingly, the equation $\Delta_N H = 0$ can be written
in these coordinates as
    $$H_{rr} + \left(\frac{n-1}{r}+\gamma(r,\theta)\right) H_r+
\frac{1}{r^2}\Delta_{S(r)}H = 0\ ,$$
where $\gamma(r,\theta)=(\sqrt{a})_r/\sqrt{a}$ with $a=\Det(a_{ij})$,
and
$\Delta_{S(r)}$ is the spherical Laplacian on the geodesic sphere of
radius $r$:
$$\Delta_{S(r)} H:= \frac{1}{\sqrt{a}}
\frac{\partial}{\partial\theta^i} \left(\sqrt{a}a^{ij}
\frac{\partial H}{\partial\theta^j}\right) $$

The following lemma relates $q(r)^{1/2}$, the spherical $L^2$-norm
of the harmonic function $H$ on an $(n-1)$-dimensional sphere of radius $r$, 
to $M_r(u_\lambda)$, the $L^{\infty}$-norm
 of the 
eigenfunction $u_\lambda$ on an $m=n-1$ dimensional \emph{ball}
 of radius~$r$.
Let $M_r(u_\lambda):=
\max_{B(p, r)} |u_\lambda(x)|$.
\begin{lemma}
\label{lem:Q-M}
Suppose $M$ is a complete Riemannian manifold with bounded geometry.
Fix $0<\alpha<1$, $\eps>0$. 
Then for all $0<r<\inj_M$,
$$
C_{\alpha,\eps}r^m (1+r\sqrt{\lambda})^{-m-\eps}M_{\alpha r}
(u_\lambda)^2\leq q(r)\leq 
C_2 r^m \me^{2r\sqrt{\lambda}}
M_r(u_\lambda)^2\ .
$$
where  $C_{\alpha, \eps}$ depends on $\alpha, \eps$ and 
the metric, and $C_2$ depends on the metric. 
\end{lemma}
\begin{proof}
Let us denote by $d\sigma(\hat{\theta})$ the standard volume form
on the unit sphere of dimension $m-1$.
\begin{multline*}
q(r)=2\int_0^r\int 
u_\lambda(\rho, \hat{\theta})^2
\cosh^2(\sqrt{\lambda}\sqrt{r^2-\rho^2}) \\
\cdot\rho^{m-1}\sqrt{b(\rho,\hat{\theta})}\frac{r}{\sqrt{r^2-\rho^2}}\, 
\dif{\hat{\theta}} \dif\rho\\ \leq 
C
M_r(u_\lambda)^2(3+\me^{2r\sqrt{\lambda}})
\int_0^r \int_{S^{m-1}}\rho^{m-1}\frac{r}{\sqrt{r^2-\rho^2}}\, 
d\sigma(\hat{\theta})d\rho \\=  C\omega_m r^m M_r(u_\lambda)^2
(3+\me^{2r\sqrt{\lambda}})\ ,
\end{multline*}
where we used the fact that the volume element 
is bounded from above by the metric 
(\cite{bish-cri64}*{Ch. 11, Th. 15}).

On the other hand, we have
$$
q(r)\geq 2 
\int_0^r\int u_\lambda(\rho,\hat{\theta})^2
 \rho^{m-1}\sqrt{b}\dif{\hat{\theta}} \dif\rho = 
 \int_{B^m(p, r)}u_\lambda^2\,\dif\Vol_M \ .$$
Hence, from elliptic regularity we get 
$$ q(r)\geq
C_{\alpha,\eps} M_{\alpha r}(u_\lambda)^2 r^m(1+r\sqrt{\lambda})^{-m-\eps}\ ,$$
where $C_{\alpha,\eps}$ depends on the metric, on $\alpha$ and on $\eps$.
\end{proof}
From Corollary~\ref{cor:integrated-convexity} and Lemma~\ref{lem:Q-M}
we find
\begin{theorem}
\label{thm:local-growth}
  Let $M$ be a complete Riemannian manifold of dimension $m$
with $|\Sec_M|\leq K$. 
Then for all $r\leq s<C/\sqrt{K}$
$$\frac{M_{3r}(u_\lambda)}{M_{2r}(u_\lambda)}
\leq C_1\me^{C_2 s\sqrt{\lambda}}
\left(\frac{M_{8s}(u_\lambda)}{M_{3 s}(u_\lambda)}\right)^{1+C_3r^2K}\ ,$$
where the constants $C_2, C_3$ denote
positive constants which depend only on the injectivity radius
of $M$, while $C_1$ depends on bounds on the metric,
its derivatives and its ellipticity constant.
\end{theorem} 
\remark\ The subindices $3r, 2r, 8s, 3s$
can be replaced by $\beta r, r, \gamma s, s$ respectively,
where $1<\beta<2$ and $\gamma>\beta$. The constants $C_2, C_3$ can
be taken to be independent of $\beta, \gamma$, while
$C_1\to \infty$ as $\gamma/\beta\to 1$. 
\section{Two global growth estimates}
\label{sec:global-growth}
In this section we deduce from the local inequality in 
Theorem~\ref{thm:local-growth} two global results in the compact case.
\subsection{Large values on large balls}
\begin{theorem}
\label{thm:large-value-large-ball}
Let $M$ be a compact Riemannian manifold of dimension $m$.
Then for all eigenfunctions $u_\lambda$ and $r>0$
$$\frac{\max_{B(x, r)} |u_\lambda|}{\max_M|u_\lambda|} 
\geq C_g(r, d_M)\me^{-C_2 d_M\sqrt{\lambda}}\quad \forall x\in M\ ,$$
where $d_M$ is the diameter of $M$,
\end{theorem}
\begin{proof}
Normalize $u_\lambda$ so $\max_M |u_\lambda|=1$.
Take $r=s$ in Theorem~\ref{thm:local-growth}. We get
\begin{equation}
\label{ineq:growth-cor}
M_{3r}(u_\lambda)^{2+C_3r^2K} \leq 
C_1\me^{C_2 r\sqrt{\lambda}}M_{8r}(u_\lambda)^{1+C_3r^2K}
M_{2r}(u_\lambda)
\leq C_1\me^{C_2 r\sqrt{\lambda}} M_{2r}(u_\lambda)\ .
\end{equation}

Let $|u_\lambda(x_0)|=1$. Fix $r_0>0$ small enough in order to apply
Theorem~\ref{thm:local-growth}.
Take a point $x$ in $M$. There exists a sequence of points 
$x_0, x_1, \ldots x_N=x$, such that $d(x_k, x_{k+1}) < r_0$,
for $0\leq k\leq N-1$, where 
$N$ only depends on $r_0$ and the diameter of $M$.
Inequality~(\ref{ineq:growth-cor}) gives
\begin{multline}
\label{ineq:one-step}
\max_{B(x_k, 2r_0)}|u_\lambda|\geq 
C_1^{-1}\me^{-C_2 r_0\sqrt{\lambda}} 
\max_{B(x_{k}, 3r_0)}|u_\lambda|^{2+C_2r^2K}\geq \\
C_1^{-1}\me^{-C_2 r_0\sqrt{\lambda}}
\max_{B(x_{k-1}, 2r_0)} |u_\lambda|^3\ .
\end{multline}

Multipliying the inequalities~(\ref{ineq:one-step}) for $1\leq k\leq N$
gives 
$$\max_{B(x, 2r_0)}|u_\lambda|
\geq C_1^{-N}\me^{-C_2 N r_0\sqrt{\lambda}} \geq
C_1^{-N}\me^{-C_2d\sqrt{\lambda}}\ .$$
\end{proof}

\subsection{Global DF growth Bound}
\begin{theorem}[\cite{don-fef88}]
\label{thm:DF-bound}
For all eigenfunctions $u_\lambda$, $x\in M$ and $r>0$
$$\frac{\max_{B(x, 3r)} |u_\lambda|}{\max_{B(x, 2r)} |u_\lambda|} 
\leq C_g(d_M)\me^{C_2d_M\sqrt{\lambda}}\ .$$
\end{theorem}
\begin{proof}
Let $R>0$ be as in Theorem~\ref{thm:large-value-large-ball}.
If $r\geq R$ the theorem follows from  
Theorem~\ref{thm:large-value-large-ball}.
Else,
  Theorems~\ref{thm:local-growth}
and~\ref{thm:large-value-large-ball} tell us
that 
$$\frac{M_{3r}(u_\lambda)}{M_{2r}(u_\lambda)}
\leq C_g \me^{C_2 R\sqrt{\lambda}} 
\left(\frac{M_{8R}(u_\lambda)}{M_{3R}(u_\lambda)}\right)^{2}\leq 
C_g(d_M)\me^{2C_2d_M\sqrt{\lambda}}\ .$$
\end{proof}

\subsection{Outline of the proof of Yau's Conjecture for real analytic manifolds}
\label{subsec:outline}
Yau's conjecture for $C^\infty$ closed compact Riemannian manifolds
is
\begin{conj}[\cite{yau82}]
\label{conj:yau}
Let $u_\lambda$ be a $\lambda$-eigenfunction on $M$.
Then,
$$C_1\sqrt{\lambda}\leq\Vol_{n-1}(\{u_\lambda=0\}) \leq C_2\sqrt{\lambda}\ ,$$
where $C_1$, $C_2$ are constants independent of $\lambda$. 
\end{conj} 

The conjecture was proved in the case of real analytic Riemannian metrics in~\cite{don-fef88}. A major ingredient of the proof
was Theorem~\ref{thm:DF-bound}. We outline here the idea:
\vspace{1ex}

\noindent\textbf{Lower bound.}
Let $B\subset M$ be a ball of radius $r=C/\sqrt{\lambda}$
such that $u_{\lambda}$ vanishes at the center of $B$. One can cover, say, $1/2$
of the volume of $M$ by a disjoint collection~$\Bl$ of such balls (\cite{bru78}).
One observes that if the growth of $u_\lambda$ in a ball $B$ 
is smaller than, say, $20$ then one can control 
from below the size of the nodal set in $B$.
 This can be seen for harmonic functions
in the unit ball using the mean value principle 
and the isoperimetric inequality, and can be adapted to
 eigenfunctions on balls of radius $C/\sqrt{\lambda}$.

The main claim is that on at least, say, $10\%$ of the balls
in the collection $\Bl$ the growth is bounded by $20$. 

We can assume $M$ is contained in one coordinate neighbourhood $U=\{|x|<30\}\subset\Rb^n$.
One can continue the function $u_\lambda$ to a holomorphic 
function~$F$ on $U\times U\subset \Cb^{n}$. 
We assume $F|_{U\times \{0\}} = u_\lambda$ and 
we set $Q\subset U\times\{0\}$ to be a Euclidean real cube.
The point is that due to Theorem~\ref{thm:DF-bound} the growth of $F^2$
in $U\times U$ is controlled by $\sqrt{\lambda}$.

We subdivide $Q$ to small sub-cubes $Q_\nu$ of sides $1/\sqrt{\lambda}$.
The next idea is that in order to bound the growth
of $F$ in a cube $Q_\nu$ by a constant independent of $\lambda$
it is enough to say that $F$ is close to its average on $Q_\nu$
for most of the points in $Q_\nu$. This property behaves well under
averaging. Therefore, it can be reduced to a dimension one problem:
 $Q=[-1,1]$, $B=|z|<2$, $F$ is 
a holomorphic function defined on $B$, $F$ is real on the real line.
and its growth is bounded by $\sqrt{\lambda}$.
First we replace $F$ by a polynomial $P$ of degree $\sqrt{\lambda}$.
One divides $Q$ into segments $Q_\nu$ of size $1/\sqrt{\lambda}$.
One has to show that $P$ is close to its average on $10\%$ of these
intervals. To that end the Hilbert transform is called.
\vspace{1ex}

\noindent\textbf{Upper bound.}
The size of the nodal set is estimated from above by Crofton's formula.
To estimate from above the number of zeros  on a real line interval
$I\subset Q$
one uses Jensen's formula in a complex line $\Cb$ containing $I$. 
For this one has to have a bound on the growth of $F$ in $U\times U$.
 
\section{Proof of Theorem~\ref{thm:perturbed-convexity}}
\label{sec:proof-thm}
\subsection{Preliminary geometric estimates}
\label{sec:eps-estimates}
Let $N$ be a Riemannian manifold of dimension $n$.
Fix a point $p$, and let $r(x)=\dist(x, p)$.
Let 
$$\gamma_K= \Delta r - (n-1)\cot_K r\ .$$
$\gamma_K$ is controlled by the curvature of $N$:
\begin{lemma}
\label{lem:gammaK}
If $\kappa\leq \Sec_N\leq K$
then $0\leq\gamma_K\leq (n-1)(\cot_{\kappa} r-\cot_K r)$.
\end{lemma}
\begin{proof}
Both parts directly follow from the Hessian Comparison Theorem
(\cites{bish-cri64, schoen-yau94}).
\end{proof}
\begin{lemma}
\label{lem:gammaK'}
Suppose $\kappa\leq\Sec_N\leq K$.
Then, we have 
$$\gamma_{K,r}\geq -(n-1)(K-\kappa)\ .$$
\end{lemma}
\begin{proof}
We know~(\cite{petersen06}*{Ch. 9.1})
$$\gamma_{K,r} = (\Delta r)_r+\frac{n-1}{(\sin_K r)^2}
=-\Ric(\partial_r, \partial_r)-\|\Hess(r)\|^2+\frac{n-1}{(\sin_K r)^2}\ .$$
By the Hessian comparison theorem (\cite{schoen-yau94})
\begin{equation}
\label{ineq:hessian}
(\cot_K r)\| X\|^2\leq
\Hess(r)(X, X)\leq (\cot_{\kappa} r)\|X\|^2 \ .
\end{equation}
Hence,
$$|\Hess(r)(X, X)|^2\leq  (\cot_\kappa r)^2\|X\|^4\ .$$  
We can choose an orthonormal basis 
$(\partial_r, e_1, \ldots, e_{n-1})$ in which 
$\Hess(r)$ is diagonalized. Then we see
$$\|\Hess(r)\|^2 =\sum |\Hess(r)(e_i, e_i)|^2 \leq
(n-1)(\cot_{\kappa} r)^2\ .$$
Consequently,
\begin{multline*}
\gamma_{K, r}\geq -(n-1)K-(n-1)(\cot_\kappa r)^2
+\frac{n-1}{(\sin_K r)^2}=(n-1) (\cot_K^2 r-\cot_\kappa^2 r) \\
\geq -(n-1)(K-\kappa)
\end{multline*}
where the last inequality follows from parts~(iii) and~(iv) of Lemma~\ref{lem:xcot2} below.
\end{proof}
\begin{lemma}
\label{lem:xcot2}
\begin{itemize}
\item[\textup{(i)}] $-1/3\leq \left(\sqrt{x}\cot\sqrt{x}\right)'\leq 0$
for all $0\leq x<(\pi/2)^2$.
\item[\textup{(ii)}] $0\leq \left(\sqrt{x}\coth\sqrt{x}\right)'\leq 1/3$
for all $x\geq 0$.
\item[\textup{(iii)}]  $-1\leq \left(x\cot^2\sqrt{x}\right)'\leq 0$ for all $0\leq x<(\pi/2)^2$ .
\item[\textup{(iv)}]  $0\leq \left(x\coth^2\sqrt{x}\right)'\leq 1$ for all $x\geq 0$
\end{itemize}
\begin{proof}
We prove the right inequality in~(ii):
Since $y\coth y\geq 1$, we have
$(3y+2y\sinh^2 y)'\geq 3(\cosh y\sinh y)'$.
Integrating, we conclude that
$$3y+2y\sinh^2 y\geq 3\cosh y\sinh y\ .$$ Equivalently,
$(y\coth y)'\leq 2y/3$. Hence, $(\sqrt{x}\coth\sqrt{x})'\leq 1/3$.

We prove the left inequality in~(iii):
\begin{equation}
\label{eqn:xcot2'}
(x\cot^2\sqrt{x})'= \cot^2\sqrt{x} -\frac{\sqrt{x}\cot\sqrt{x}}{\sin^2\sqrt{x}}\ .
\end{equation}
Observe that for $0\leq y<\pi/2$
\begin{equation}
\label{eqn:ycoty}
y\cot y \leq 1 \ .
\end{equation}
From~(\ref{eqn:xcot2'}) and~(\ref{eqn:ycoty}) it follows that 
$$(x\cot^2\sqrt{x})'\geq \cot^2\sqrt{x} - \frac{1}{\sin^2{\sqrt{x}}} = -1\ .$$
The proofs of the all other inequalities in the Lemma are omitted.
\end{proof}
\end{lemma}
\subsection{Choice of coordinates and notations}
We take geodesic polar coordinates centred at $p\in N$.
Fix any $K\in\Rb$. The metric can be written as
$$g=\dif r^2+(\sin_K r)^2 (a_K)_{ij}\dif\theta^i\dif\theta^j\ ,$$
where $\theta^i$ are coordinates on the standard unit sphere 
$S^{n-1}\subset\Rb^n$.

We denote the determinant of the matrix $(a_K)_{ij}$ by $a_K$.
The Laplacian on $N$ can be written as
$$(\Delta f)(r, \theta) = f_{rr}(r,\theta)+((n-1)\cot_K r+\gamma_K) f_r(r,\theta) + \frac{1}{(\sin_K r)^2}\left(\Delta_S f(r,\cdot)\right)(\theta)\ ,$$
where $\Delta_S$ is the following operator acting on functions $g$
defined on $S^{n-1}$:

$$(\Delta_S g)(\theta):= \frac{1}{\sqrt{a_K}}\frac{\partial}{\partial\theta^i}
\left(a_K^{ij}\sqrt{a_K}\frac{\partial g}{\partial\theta^j}\right)\ .$$

We emphasize that the definition of $\Delta_S$ depends on our choice of $K$. 
With these definitions we also have
$$\gamma_K=\frac{(\sqrt{a_K})_r}{\sqrt{a_K}}\ .$$

\subsection{Proof of part~(i)}
We observe that
$$q(r)=\int u^2 (\sin_K r)^{n-1} \sqrt{a_K}\dif\theta\ ,$$
where the integration is understood to be performed over the parameter space $[0,\pi]^{n-2}\times[0,2\pi]$
for $S^{n-1}$ in $\Rb^{n-1}$.
A straightforward computation shows
\begin{lemma}
\label{lem:q'}
\begin{multline*}
q'(r)=\int 2uu_r(\sin_K  r)^{n-1}\sqrt{a_K}\dif\theta +
\int u^2 \gamma_K(\sin_K r)^{n-1} \sqrt{a_K}\dif\theta\\
+(n-1) (\cot_K r) \int u^2 (\sin_K r)^{n-1} \sqrt{a_K}\dif\theta \ .
\end{multline*}
\end{lemma}
\begin{lemma}
\label{lem:uur>0}
$$\int 2uu_r (\sin_K r)^{n-1}\,\sqrt{a_K}\dif\theta \geq 0\ .$$
\end{lemma}
\begin{proof}
By Green's formula and the harmonicity of $u$
\begin{multline*}
\int 2uu_r (\sin_K r)^{n-1}\sqrt{a_K}\,\dif\theta
=\int_{\partial B(p, r)} \frac{\partial (u^2)}{\partial\hat{n}}\dif A_r \\
=\int_{B(p, r)} \Delta(u^2)\,\dif\Vol= \int_{B(p, r)} 2|\nabla u|^2\, \dif\Vol \ .
\end{multline*}
\end{proof}
\begin{proof}[Proof of Theorem~\ref{thm:perturbed-convexity}, part~(i)]
Part~(i) of the theorem follows directly from
Lemma~\ref{lem:q'}, Lemma~\ref{lem:uur>0} and Lemma~\ref{lem:gammaK}.
\end{proof}
\subsection{Proof of part~(ii)}
Let $w=(\sin_K r)^l u$, where $l=(n-2)/2$.
$w$ satisfies the equation
\begin{equation}
\label{eqn:w-eqn}
w_{rr}+(\cot_K r+\gamma_K) w_r +l(l+1)Kw -\frac{l^2w}{(\sin_K r)^2}
+\frac{\Delta_S w}{(\sin_K r)^2} =0\ .
\end{equation}
Let 
\begin{equation}
Q(r) = \int w(r,\theta)^2\, \sqrt{a_K}\dif\theta=\frac{q(r)}{\sin_K(r)}\ .
\end{equation}
Let us also set 
$$\nabla_S w:= (\sin_K r)\left(\nabla w - w_r\frac{\partial}{\partial
  r}\right)=\frac{1}{\sin_K r} a_K^{ij}
\frac{\partial w}{\partial\theta^i}\frac{\partial}{\partial\theta^j}\ .$$

$\nabla_S$ is defined in this way in order to have Green's formula
\begin{equation}
\label{identity:green}
\int f(\theta)(\Delta_S g)(\theta)\sqrt{a_K}\dif\theta
=-\int \langle\nabla_S f,\nabla_S g\rangle \sqrt{a_K}\dif\theta\ .
\end{equation}
Note also that $\langle\nabla_S w, \partial_r\rangle=0$.

\begin{lemma}
\label{lem:Q'}
\begin{itemize}
\item[\textup{(i)}] $Q'(r) = 
\int 2w(w_r+\gamma_K w/2)\sqrt{a_K}\dif\theta$.
\item[\textup{(ii)}] $Q'(r)\geq (n-2)(\cot_K r)Q(r)\geq 0$.
\end{itemize}
\end{lemma}
\begin{proof}
Part~(i) is a direct calculation.
Part~(ii) is just another formulation of part~(i) of 
Theorem~\ref{thm:perturbed-convexity}.
\end{proof}
A second direct calculation using equation~(\ref{eqn:w-eqn}) 
and formula~(\ref{identity:green}) gives
\begin{lemma}
\label{lem:Q''+Q'}
\begin{multline*}
Q''(r)+(\cot_K r) Q'(r) =
2\int \left(w_r+\frac{\gamma_K}{2}w\right)^2\sqrt{a_K}\dif\theta \\
+\frac{2}{(\sin_K r)^2}\int |\nabla_S w|^2\dif\theta 
+\frac{2l^2}{(\sin_K r)^2}\int w^2\sqrt{a_K}\dif\theta\\
-2l(l+1)K\int w^2\sqrt{a_K}\dif\theta 
+ \int w^2\left(\gamma_{K, r}+\gamma_K\cot_K r+\frac{\gamma_K^2}{2}\right)\sqrt{a_K}\dif\theta\ .
\end{multline*}
\end{lemma}
\begin{lemma}
\label{lem:Q''+Q'+Q}
\begin{multline*}
Q''(r)+(\cot_K r) Q'(r) \geq
2\int \left(w_r+\frac{\gamma_K}{2}w\right)^2\sqrt{a_K}\dif\theta \\
+\frac{2}{(\sin_K r)^2}\int |\nabla_S w|^2\sqrt{a_K}\dif\theta 
+\frac{2l^2}{(\sin_K r)^2}Q
-2l(l+1)KQ 
-(n-1)(K-\kappa)Q\ .
\end{multline*}
\end{lemma}
\begin{proof}
This estimate is due to Lemma~\ref{lem:Q''+Q'} and
the estimates on $\gamma_K$ and $\gamma_{K, r}$ in 
Lemma~\ref{lem:gammaK} and Lemma~\ref{lem:gammaK'} respectively.
\end{proof}
Immediately we get
\begin{lemma}
\label{lem:logq''}
\begin{multline*}
Q''(r)+(\cot_K r)Q'(r)-\frac{Q'(r)^2}{Q(r)}
\geq 2\int \left(w_r+\frac{\gamma_K}{2}w\right)^2\sqrt{a_K}\dif\theta \\
+\frac{2}{(\sin_K r)^2}\int |\nabla_S w|^2\sqrt{a_K}\dif\theta
+\frac{2l^2}{(\sin_K r)^2}Q
-2l(l+1)KQ - (n-1)(K-\kappa)Q \\
-\frac{4\left(\int w(w_r+\gamma_K w/2)\,\sqrt{a_K}\dif\theta\right)^2}
{\int w^2\,\sqrt{a_K}\,\dif\theta}
\end{multline*}
\end{lemma}
%
\begin{lemma}
\label{lem:halfCS}
\begin{multline*}
Q''(r)+(\cot_K r)Q'(r)-\frac{Q'(r)^2}{Q(r)}+(n-1)(\cot_\kappa r-\cot_K r)Q'(r) \\
\geq \frac{\vphi(r)}{(\sin_K r)^2}
+\frac{2l^2}{(\sin_K r)^2}Q-2l(l+1)KQ - (n-1)(K-\kappa)Q 
\end{multline*}
where
\begin{equation*}
\vphi(r)=
-2(\sin_K r)^2\int w_r^2\,\sqrt{a_K}\dif\theta
+2\int |\nabla_S w|^2\,\sqrt{a_K}\dif\theta\ .
\end{equation*}
\end{lemma}
\begin{proof}
\begin{multline*}
Q''(r)+(\cot_K r)Q'(r)-\frac{Q'(r)^2}{Q(r)}
\\\geq
 2\int \left(w_r+\frac{\gamma_K}{2}w\right)^2\sqrt{a_K}\dif\theta \\
+\frac{2}{(\sin_K r)^2}\int |\nabla_S w|^2\sqrt{a_K}\dif\theta
+\frac{2l^2}{(\sin_K r)^2}Q\\
-2l(l+1)KQ - (n-1)(K-\kappa)Q \\
-\frac{2\left(\int w(w_r+\gamma_K w/2)\,\sqrt{a_K}\dif\theta\right)^2}
{\int w^2\,\sqrt{a_K}\,\dif\theta}
-\frac{2\left(\int w(w_r+\gamma_K w/2)\,\sqrt{a_K}\dif\theta\right)^2}
{\int w^2\,\sqrt{a_K}\,\dif\theta}\\
\geq  
\frac{2}{(\sin_K r)^2}\int |\nabla_S w|^2\sqrt{a_K}\dif\theta
+\frac{2l^2}{(\sin_K r)^2}Q 
-2l(l+1)KQ \\- (n-1)(K-\kappa)Q 
-\frac{2\left(\int ww_r\sqrt{a_K}\dif\theta+\int\gamma_K w^2/2\,\sqrt{a_K}\dif\theta\right)^2}
{\int w^2\,\sqrt{a_K}\,\dif\theta}\\
\geq
 \frac{\vphi(r)}{(\sin_K r)^2}+\frac{2l^2}{(\sin_K r)^2}Q -\frac{\int\gamma_K w^2\sqrt{a_K}\dif\theta}{\int w^2\sqrt{a_K}\dif\theta}
Q'
+\frac{(\int \gamma_K w^2\sqrt{a_K}\dif\theta)^2}{2\int w^2\sqrt{a_K}\dif\theta}\\
-2l(l+1)KQ - (n-1)(K-\kappa)Q\\
\geq 
\frac{\vphi(r)}{(\sin_K r)^2}+\frac{2l^2}{(\sin_K r)^2}Q \\
-(n-1)(\cot_\kappa r-\cot_K r)Q'
-2l(l+1)KQ - (n-1)(K-\kappa)Q\ .
\end{multline*}
The first inequality is just a rewriting of Lemma~\ref{lem:logq''}.
In the second inequality we applied Cauchy-Schwarz inequality on
the last term. In the third inequality we unfolded the parentheses
in the last term and applied Cauchy-Schwarz inequality on
the term $\int ww_r\sqrt{a_K}\dif\theta$. In the last inequality
we used the fact that $Q'\geq 0$ (Lemma~\ref{lem:Q'}) and the
estimates on $\gamma_K$ in Lemma~\ref{lem:gammaK}. 
\end{proof}
It remains to control the function $\vphi$ in terms of $Q$ and $Q'$.
We would like first to calculate the derivative of $\vphi$.
To that end, we
recall the definition and some of the properties 
of the Hessian as a bilinear form: 
$$\mathrm{Hess} f(X, Y) := XYf-(\nabla_X Y) f = 
\langle Y, \nabla_X \grad f\rangle =\langle X, \nabla_Y\grad f\rangle\ .$$

In a geodesic ball centred at $p$, 
we have a radial field $\grad\, r =\partial_r$, tangent
to the geodesics emanating from $p$.
Since $\partial_r$
is tangent to a geodesic, we have $\nabla_{\partial_r}\partial_r =0$.
As a consequence $(\mathrm{Hess}\, r)(\partial_r, Y)=0$ for all vectors $Y$.
When computing the derivative of $\vphi$, it is convenient to have
the following formula:
\begin{lemma}
\label{lem:formula}
$$(|\nabla_S f|^2)_r
=2\langle\nabla_S f, \nabla_S f_r\rangle
-2\Hess(r)(\nabla_S f, \nabla_S f)
+2(\cot_K r)|\nabla_S f|^2$$
\end{lemma}
\begin{proof}
\begin{multline*}
2\Hess(r)(\nabla_S f, \nabla_S f) =
2(\sin_K r)^2\Hess(r)(\nabla f, \nabla f) \\=
2(\sin_K r)^2\langle \nabla f,\nabla_{\nabla f}\partial_r\rangle
=2(\sin_K r)^2\langle \nabla f, \nabla_{\partial_r} \nabla f 
+[\nabla f, \partial_r]\rangle \\
=(\sin_K r)^2(|\nabla f|^2)_r
+2(\sin_K r)^2[\nabla f, \partial_r]f \\
=-(\sin_K r)^2(|\nabla f|^2)_r+2(\sin_K r)^2\langle \nabla f,\nabla f_r \rangle \\
=-(\sin_K r)^2\left(f_r^2+(\sin_K r)^{-2}|\nabla_S f|^2\right)_r
+2(\sin_K r)^2f_rf_{rr}
\\+2\langle \nabla_S f, \nabla_S f_r\rangle
=-(|\nabla_S f|^2)_r
+2(\cot_K r)|\nabla_S f|^2+2\langle \nabla_S f, \nabla_S f_r \rangle
\end{multline*}
\end{proof}
Using the formula in Lemma~\ref{lem:formula}
we can readily compute the derivative of $\vphi(r)$ (defined in 
Lemma~\ref{lem:halfCS}):
\begin{lemma}
\label{lem:phi'}
\begin{multline*}
\vphi'(r) = -4\int \Hess(r)(\nabla_S w, \nabla_S w)\,
\sqrt{a_K}\dif\theta \\
+4(\cot_K r)\int |\nabla_S w|^2\,\sqrt{a_K}\dif\theta
+2l(l+1)K(\sin_K r)^2Q'-2l^2Q' \\
+2(\sin_K r)^2\int |\nabla w|^2\gamma_K\sqrt{a_K}\dif\theta \\
+2l^2\int w^2\gamma_K\sqrt{a_K}\dif\theta
-2l(l+1)K\sin_K^2 r\int w^2\gamma_K\sqrt{a_K}\dif\theta
\end{multline*}
\end{lemma}
\begin{lemma}
\label{lem:phi'bound}
\begin{multline*}
\vphi'(r)\geq  -4(\cot_\kappa r-\cot_K r)
\int |\nabla_S w|^2\sqrt{a_K}\dif\theta
+2l(l+1)K(\sin_K r)^2Q'
-2l^2Q'\\
-2l(\cot_\kappa r-\cot_K r)K^{+}(\sin_K r)^2 Q\ .
\end{multline*}
\end{lemma}
\begin{proof}
This is due to inequality~(\ref{ineq:hessian}) and Lemma~\ref{lem:phi'}.
\end{proof}
In Lemma~\ref{lem:phi-bound} we integrate the inequality in Lemma~\ref{lem:phi'bound}. We need a few lemmas before that:
\begin{lemma}
\label{lem:int-sin-q'}
$$\left(1-\frac{2}{n}\right)(\sin_K r)^2Q(r)\leq\int_0^r(\sin_K\rho)^2Q'(\rho)\dif\rho \leq (\sin_K r)^2Q(r)\ .$$
\end{lemma}
\begin{proof}
The RHS follows from the fact that $\sin_K\rho$ is monotonically
increasing in $\rho$ and $Q'\geq 0$.
By derivating the LHS we see that it is enough to prove
\begin{equation}
\label{ineq:lhs'}
\left(1-\frac{2}{n}\right)(\sin_K r)^2(2(\cot_K r) Q(r)+
Q'(r))\leq (\sin_K r)^2Q'(r)\ .
\end{equation}
Inequality~(\ref{ineq:lhs'}) is equivalent to part~(ii) of Lemma~\ref{lem:Q'}.
\end{proof}
\begin{lemma}
$$\int_0^r\int|\nabla_S w|^2\sqrt{a_K}\dif\theta\dif\rho
\leq\frac{(\sin_K r)^2}{2}\left(Q'(r)-(n-2)(\cot_K r)Q\right)\ .$$
\end{lemma}
\begin{proof}
\begin{multline*}
\int_0^r\int|\nabla_S w|^2\sqrt{a_K}\dif\theta\dif\rho
=\int_0^r\int|\nabla_S u|^2(\sin_K\rho)^{n-2}
\sqrt{a_K}\dif\theta\dif\rho\\
\leq\int_0^r\int|\nabla u|^2(\sin_K\rho)^n\sqrt{a_K}\dif\theta\dif\rho
\\
\leq\sin_K r\int_0^r\int|\nabla u|^2(\sin_K\rho)^{n-1}\sqrt{a_K}\dif\theta\dif\rho
\\ =\sin_K r\int_{B(p, r)} |\nabla u|^2\dif\Vol
=\sin_K r\int uu_r(\sin_K r)^{n-1}\sqrt{a_K}\dif\theta\\
=(\sin_K r)^2\int ww_r\sqrt{a_K}\dif\theta-l\cot_K r(\sin_K r)^2\int w^2\sqrt{a_K}\dif\theta\\
=(\sin_K r)^2\int w(w_r+\gamma_K w/2)\sqrt{a_K}\dif\theta
-l(\cot_K r)(\sin_K r)^2\int w^2\sqrt{a_K}\dif\theta\\
-(\sin_K r)^2\int w^2\gamma_K/2\sqrt{a_K}\dif\theta
\leq\frac{(\sin_K r)^2}{2}\left(Q'(r)-(n-2)(\cot_K r)Q\right)\ .
\end{multline*}
\end{proof}
\begin{lemma}
\label{lem:int-sin-q}
$$\int_0^r (\sin_K \rho)^2 Q(\rho)\dif\rho 
\leq r(\sin_K r)^2Q(r)\ . $$
\end{lemma}
\begin{proof}
$\sin_K\rho$ and $Q(\rho)$ are monotonically increasing in $\rho$.
\end{proof}
\begin{lemma}
\label{lem:phi-bound}
\begin{multline*}
\frac{\vphi(r)}{(\sin_K r)^2}\geq 
-2(\cot_\kappa r-\cot_K r)\left(Q'-(n-2)(\cot_K r) Q\right)\\
+\frac{n(n-2)}{2}KQ-(n-2)K^+Q-\frac{(n-2)^2Q}{2(\sin_K r)^2}
-(n-2)(\cot_\kappa r-\cot_K r)rK^{+}Q.
\end{multline*}
\end{lemma}
\begin{proof}
Observe that the functions
$\cot_\kappa r-\cot_K r$ and $\sin_K r$ are both
 monotonically increasing.
Hence, integrating Lemma~\ref{lem:phi'bound}, applying
Lemmas \ref{lem:int-sin-q'}--\ref{lem:int-sin-q} we obtain
\begin{multline*}
\vphi(r)\geq-4(\cot_\kappa r-\cot_K r)
\int_0^r\int|\nabla_S w|^2\sqrt{a_K}\dif\theta\dif\rho\\
+2l(l+1)K\int_0^r (\sin_K\rho)^2Q'(\rho)\dif\rho
-2l^2Q\\
-2l(\cot_\kappa r-\cot_K r)K^{+}\int_0^r(\sin_K\rho)^2Q(\rho)\dif\rho\\
\geq -2(\sin_K r)^2(\cot_\kappa r-\cot_K r)\left(Q'
-(n-2)(\cot_\kappa r-\cot_K r)(\cot_K r)Q\right)\\
+2l(l+1)K(\sin_K r)^2Q(r)-2l(l+1)K^+(2/n)(\sin_K r)^2Q-2l^2Q\\
-(n-2)(\cot_\kappa r-\cot_K r)rK^+(\sin_K r)^2Q(r)\ .
\end{multline*}
\end{proof}
\begin{proof}[Proof of Theorem~\ref{thm:perturbed-convexity}, part~(ii)]
From Lemma~\ref{lem:halfCS} and Lemma~\ref{lem:phi-bound}
we get
\begin{multline}
\label{ineq:almost-thm}
Q''(r)+(\cot_K r)Q'(r)-\frac{Q'(r)^2}{Q(r)}+(n-1)(\cot_\kappa r-\cot_K r)Q'(r) \\
\geq -2(\cot_\kappa r-\cot_K r)Q'(r)
+2(n-2)(\cot_\kappa r-\cot_K r)(\cot_K r)Q\\
+\frac{n(n-2)}{2}KQ(r)
-(n-2)K^+Q-\frac{(n-2)^2}{2(\sin_K r)^2}Q-(n-2)(\cot_\kappa r-\cot_K r)rK^+Q\\-\frac{n(n-2)}{2}KQ+
\frac{(n-2)^2}{2(\sin_K r)^2}Q
 - (n-1)(K-\kappa)Q=\\
  -2(\cot_\kappa r-\cot_K r)Q'(r)
+2(n-2)(\cot_\kappa r-\cot_K r)(\cot_K r)Q\\
-\!(n-2)K^+Q-\!(n-2)\!(\cot_\kappa r-\cot_K r)rK^+Q
 -\! (n-1)(K-\kappa)Q\ .
\end{multline}
We get
\begin{multline}
\label{ineq:Q-thm}
(\log Q)''(r)+(\cot_K r)(\log Q)'(r) 
+(n+1)(\cot_\kappa r-\cot_K r)(\log Q)'(r)\\ \geq 
-(n-1)(K-\kappa)-(n-2)K^{+}-(n-2)(K-\kappa)\frac{r^2 K^+}{3}
\\ \geq
-(2n-3)(K-\kappa)-(n-2)K^{+}\ ,
\end{multline}
where we applied parts~(i) and~(ii) of Lemma~\ref{lem:xcot2}.
Recall $q(r)=Q(r)(\sin_K r)$. A direct
computation shows 
\begin{multline}
\label{identity:sin-thm}
(\log\sin_K r)''+(\cot_K r)(\log\sin_K r)'+(n+1)(\cot_\kappa r-\cot_K r)
(\log\sin_K r)'
\\=-K+(n+1)\cot_K r(\cot_\kappa r-\cot_K r)\geq -K\ .
\end{multline}
Finally, adding up~(\ref{ineq:Q-thm}) and~(\ref{identity:sin-thm})  gives
the statement in the theorem.
\end{proof} 
\subsection{Proof of Corollary~\ref{cor:integrated-convexity}}
\label{subsec:cor}
\begin{proof}[Proof of Corollary~\ref{cor:integrated-convexity}]
From Theorem~\ref{thm:perturbed-convexity}
\begin{equation}
\label{ineq:temp}
q''(r)+(\cot_K r)q'+(n+1)(\cot_{-K} r-\cot_K r)q'(r)\geq -(5n-7)Kq
\end{equation}
From Lemma~\ref{lem:xcot2} and from the fact that $q'\geq 0$ (part~(i)
of Theorem~\ref{thm:perturbed-convexity}) we know that 
\begin{equation}
\label{ineq:temp2}
(n+1)(\cot_{-K} r-\cot_K r)q' \leq 2(n+1)rKq'/3\ .
\end{equation}

From part~(i) of Theorem~\ref{thm:perturbed-convexity}
 we know that 
\begin{equation}
\label{ineq:temp3}
-(5n-7)Kq\geq -5Kq'(r)/\cot_K r\ .
\end{equation}
 It is easy to check that  
 $1/\cot_K r \leq 2r$ for $r\leq\pi/(3\sqrt{K})$.
 
Hence, from inequalities~(\ref{ineq:temp})--(\ref{ineq:temp3}) we get
\begin{equation}
\label{ineq:approx-thm}
q''(r)+\frac{1+8n r^2K}{r}q'-\frac{q'(r)^2}{q(r)}\geq 0
\end{equation}
 for $r\sqrt{K}<\pi/3$.
If we define $l(t)=q(\me^{t})$ then~(\ref{ineq:approx-thm})
is equivalent to 
\begin{equation}
\label{ineq:approx-thm-t}
l''(t)+8nK\me^{2t}l'(t)\geq 0
\end{equation}
for $t<-(\log K)/2+\log (\pi/3)$\ .
We will now integrate inequality~(\ref{ineq:approx-thm-t}).



Inequality~(\ref{ineq:approx-thm-t}) can be rewritten as 
$(\me^{4nK\me^{2t}}l'(t))'\geq 0$,
from which we see that for $s_2<s_1$
\begin{equation}
\label{ineq:l'}
l'(s_2)\leq \me^{4nK(\me^{2s_1}-\me^{2s_2})}l'(s_1)
\leq \me^{4nK\me^{2s_1}}l'(s_1)\ ,
\end{equation}
where, the last inequality is true since $l'(s)\geq 0$
from part~(i) of Theorem~\ref{thm:perturbed-convexity}.
Hence for $t_2<t_1$ such that $16nK\me^{2t_1}<1$, 
and $0\leq h\leq \log 2$
\begin{multline*}
   l(t_2+h)-l(t_2)=\int_0^{h} l'(t_2+s)\, ds
\leq \int_0^h \me^{4nK\me^{2t_1+2s}}l'(t_1+s)\, ds \\
\leq \me^{4nK\me^{2t_1+2h}}(l(t_1+h)-l(t_1))
\leq (1+32nK\me^{2t_1})(l(t_1+h)-l(t_1))\ .
\end{multline*}
The last inequality follows from $\me^x\leq 1+2x$ for $0\leq x\leq 1$.

Going back from the variable $t$ to the variable $r$ we obtain
the stated corollary.
\end{proof}

\section{The case of constant curvature manifolds}
\label{sec:examples}
We give a new proof of Theorem~\ref{thm:log-convexity-rn} 
and a second proof of Theorem~\ref{thm:perturbed-convexity}
in the case of constant nonzero curvature in dimension two.
\subsection{Zero curvature}
Let $u_l(r,\theta) = r^l\cos(l\theta)$, $v_l=r^l\sin(l\theta)$.
$q_{u_l}(r) = q_{v_l}(r)= \pi r^{2l+1}$.
It is obvious that $\log q_l$ is a convex function of $\log r$.

Now, any harmonic function can be written
as
$$u=a_0+\sum_{l=1}^{\infty} a_l u_l(r,\theta) +b_l v_l(r,\theta)\ .$$
The functions $u_l(r, \cdot), v_l(r, \theta)$ 
are pairwise orthogonal as functions on the unit circle for all fixed $r$.
For any two orthogonal functions $f, g$
 on the unit circle for all fixed $r$ we have
$q_{f+g}(r)=q_f(r)+q_g(r)$.
We also know that the sum of log-convex functions is log-convex
and the pointwise limit of log-convex functions is log-convex.
These considerations give a short new proof of Theorem~\ref{thm:log-convexity-rn}.
\vspace{1ex}

\noindent\remark\ A similar argument carries out also in dimensions $\geq 3$.

\subsection{Positive curvature, dimension two}
The metric on the 2-dimensional sphere of constant curvature $K>0$
is given by
$$ds^2=dr^2+(\sin_K r)^2\dif\theta^2\ .$$
Here $0\leq r<\pi/\sqrt{K}$, and $0\leq \theta\leq 2\pi$.
Hence, $$q_u^K(r)=\int_0^{2\pi} u(r,\theta)^2 
(\sin_K r)\dif\theta. $$
We define also 
$q_f^0(r) = \int_0^{2\pi} f(r,\theta)^2 r\,\dif\theta$
for function defined on $\Rb^2$.

Let $f(r, \theta)$ be defined on $\Rb^2$ by
$u(r,\theta) = f(\tan (r\sqrt{K}/2), \theta)$.
$f$ is related to $u$ by a stereographic projection.
Since harmonic functions are preserved under conformal
transformations in dimension two,
$f(r, \theta)$ is harmonic if and only if $u(r, \theta)$ is harmonic.
We also note the relation
$$q_u^K(r) = \frac{q_f^0(\tan(r\sqrt{K}/2))}{\tan (r\sqrt{K}/2)}\sin_K r\ .$$

Suppose now $f$ is harmonic. Then,
from the fact that $\log q_f^0$ is a convex function of $\log r$,
we obtain
\begin{theorem}
If $K>0$ then
$$(\log q_u^K)''(r) +(\cot_K r) (\log q_u^K)'(r) \geq -K\ .$$
\end{theorem}

\subsection{Negative curvature}
In the spherical example one can replace all trigonometric functions
by the corresponding hyperbolic functions and obtain
\begin{theorem}
If $K<0$ then
$$(\log q_u^K)''(r) +(\cot_K r) (\log q_u^K)'(r) 
\geq -K > 0\ .$$
\end{theorem}

\section{Discussion}
\label{sec:discussion}
We raise several questions which we find 
interesting to pursue.
\subsection{Beyond the injectivity radius}
It would be interesting to understand whether
 Theorem~\ref{thm:perturbed-convexity} remains
true beyond the injectivity radius as long as $r\sqrt{K^{+}}<\pi/2$
in the spirit of Bishop-Gromov's Volume Comparison Theorem~(\cite{gromov-structures79}).
\vspace{-2ex}
\subsection{Proof by an orthogonal basis of functions.}
\label{subsec:orthogonal}
In a manifold of constant curvature $K\neq 0$ of dimension $\geq 3$ 
we would like to have a simple proof, inspired from the proof 
presented in section~\ref{sec:examples} for the case $K=0$. 
This would shed light
also on the sharpness of Theorem~\ref{thm:perturbed-convexity}
in dimensions $n\geq 3$.
\vspace{-2ex}
\subsection{Ricci curvature.}
Can one of the bound assumptions on the sectional curvature in Theorem~\ref{thm:perturbed-convexity} be relaxed to
 a bound on the Ricci curvature? 
 \vspace{-2ex}
\subsection{Eigenfunctions on negatively curved manifolds.}
Can we replace the extension procedure described in Section~\ref{sec:extension} by a procedure which will give us more information on the growth of eigenfunctions on negatively curved manifolds?
\vspace{-2ex}
\subsection{A comparison theorem for positive harmonic functions}
Let $f(\theta)$ be a $2\pi$-periodic non-negative function.
Let $u$ be a solution of the Dirichlet problem in the unit disk:
$\Delta u=0$ with $u(1,\theta)=f(\theta)$.
Now, suppose we consider the unit geodesic disk in a 
Riemannian manifold with non-positive variable curvature,
and solve the Dirichlet problem there. We get a solution $v(r, \theta)$.
Can we compare the values of $u$ to the values of $v$?
Or equivalently, can we compare the Poisson kernels involved?

\vfill
\noindent Dan Mangoubi,\\ 
Einstein Institute of Mathematics,\\
Hebrew University, Givat Ram,\\
Jerusalem 91904,\\
Israel\\
%
\smallskip
\texttt{\small mangoubi@math.huji.ac.il}

\vspace{-2ex}
\begin{bibdiv}
\begin{biblist}

\bib{agmon66}{book}{
   author={Agmon, Shmuel},
   title={Unicit\'e et convexit\'e dans les probl\`emes diff\'erentiels},
   series={S\'eminaire de Math\'e\-matiques Sup\'er\-ieures, No. 13 (\'Et\'e,
   1965)},
   publisher={Les Presses de l'Universit\'e de Montr\'eal, Montreal, Que.},
   date={1966},
   pages={152},
}


\bib{ahlfors-ca}{book}{
   author={Ahlfors, Lars V.},
   title={Complex analysis},
   edition={3},
   note={An introduction to the theory of analytic functions of one complex
   variable;
   International Series in Pure and Applied Mathematics},
   publisher={McGraw-Hill Book Co.},
   place={New York},
   date={1978},
   pages={xi+331},
}

\bib{almgren79}{article}{
   author={Almgren, Frederick J., Jr.},
   title={Dirichlet's problem for multiple valued functions and the
   regularity of mass minimizing integral currents},
   book={
      publisher={North-Holland},
      place={Amsterdam},
   },
   date={1979},
   pages={1--6},
}

\bib{bish-cri64}{book}{
   author={Bishop, Richard L.},
   author={Crittenden, Richard J.},
   title={Geometry of manifolds},
   series={Pure and Applied Mathematics, Vol. XV},
   publisher={Academic Press},
   place={New York},
   date={1964},
   pages={ix+273},
}

\bib{bru78}{article}{
   author={Br{\"u}ning, Jochen},
   title={\"Uber Knoten von Eigenfunktionen des Laplace-Beltrami-Operators},
   journal={Math. Z.},
   volume={158},
   date={1978},
   number={1},
   pages={15--21},
   issn={0025-5874},
}

\bib{don-fef88}{article}{
   author={Donnelly, Harold},
   author={Fefferman, Charles},
   title={Nodal sets of eigenfunctions on Riemannian manifolds},
   journal={Invent. Math.},
   volume={93},
   date={1988},
   number={1},
   pages={161--183},
   issn={0020-9910},
}

\bib{gar-lin86}{article}{
   author={Garofalo, Nicola},
   author={Lin, Fang-Hua},
   title={Monotonicity properties of variational integrals, $A_p$ weights
   and unique continuation},
   journal={Indiana Univ. Math. J.},
   volume={35},
   date={1986},
   number={2},
   pages={245--268},
}

\bib{gromov-structures79}{book}{
   author={Gromov, Mikhael},
   title={Structures m\'etriques pour les vari\'et\'es riemanniennes},
   language={French},
   series={Textes Math\'ematiques [Mathematical Texts]},
   volume={1},
   publisher={CEDIC},
   place={Paris},
   date={1981},
   pages={iv+152},
   isbn={2-7124-0714-8},
}

\bib{han-lin}{article}{
   author={Han, Qing},
   author={Lin, Fang-Hua},
   title={Nodal sets of solutions of elliptic differential equations},
   eprint={http://nd.edu/~qhan/nodal.pdf},
}

\bib{jer-leb96}{article}{
   author={Jerison, David},
   author={Lebeau, Gilles},
   title={Nodal sets of sums of eigenfunctions},
   conference={
      title={Harmonic analysis and partial differential equations (Chicago,
      IL, 1996)},
   },
   book={
      series={Chicago Lectures in Math.},
      publisher={Univ. Chicago Press},
      place={Chicago, IL},
   },
   date={1999},
   pages={223--239},
}

\bib{landis63}{article}{
   author={Landis, E. M.},
   title={Some questions in the qualitative theory of second-order elliptic
   equations (case of several independent variables)},
   journal={Uspehi Mat. Nauk},
   volume={18},
   date={1963},
   number={1 (109)},
   pages={3--62},
   translation={journal={Russian Math. Surveys},
                volume={18},
                date={1963},
                pages={1--62},
                },
   issn={0042-1316},
}

\bib{lin91}{article}{
   author={Lin, Fang-Hua},
   title={Nodal sets of solutions of elliptic and parabolic equations},
   journal={Comm. Pure Appl. Math.},
   volume={44},
   date={1991},
   number={3},
   pages={287--308},
   issn={0010-3640},
}

\bib{naz-pol-sod05}{article}{
   author={Nazarov, F{\"e}dor},
   author={Polterovich, Leonid},
   author={Sodin, Mikhail},
   title={Sign and area in nodal geometry of Laplace eigenfunctions},
   journal={Amer. J. Math.},
   volume={127},
   date={2005},
   number={4},
   pages={879--910},
   issn={0002-9327},
}

\bib{petersen06}{book}{
   author={Petersen, Peter},
   title={Riemannian geometry},
   series={Graduate Texts in Mathematics},
   volume={171},
   edition={2},
   publisher={Springer},
   place={New York},
   date={2006},
   pages={xvi+401},
   isbn={978-0387-29246-5},
   isbn={0-387-29246-2},
}


\bib{schoen-yau94}{book}{
   author={Schoen, R.},
   author={Yau, S.-T.},
   title={Lectures on differential geometry},
   series={Conference Proceedings and Lecture Notes in Geometry and
   Topology, I},
   publisher={International Press},
   place={Cambridge, MA},
   date={1994},
   pages={v+235},
   isbn={1-57146-012-8},
}

\bib{yau82}{article}{
   author={Yau, Shing-Tung},
   title={Problem section},
   conference={
      title={Seminar on Differential Geometry},
   },
   book={
      series={Ann. of Math. Stud.},
      volume={102},
      publisher={Princeton Univ. Press},
      place={Princeton, N.J.},
   },
   date={1982},
   pages={669--706},
}
\end{biblist}
\end{bibdiv}
\end{document}